\documentclass{elsarticle}

\makeatletter
\def\ps@pprintTitle{%
 \let\@oddhead\@empty
 \let\@evenhead\@empty
 \def\@oddfoot{}%
 \let\@evenfoot\@oddfoot}
\makeatother
\usepackage{graphicx}
\usepackage{hyperref}
\usepackage{amsfonts}
\usepackage{epstopdf}
\usepackage{algorithmic}
\usepackage{amssymb,amsmath,epsfig,multirow}
\usepackage[displaymath, mathlines,running]{lineno}
\usepackage{mathtools} 
\usepackage{color}
\usepackage{lineno,hyperref}
\usepackage{rotating}
\usepackage{makeidx}
\usepackage{lscape}
\usepackage{float}
\usepackage{epsfig}
\usepackage{amsmath,amssymb}
\usepackage{enumitem}
\usepackage{epic,eepic}
\usepackage{overpic}
\usepackage{color}
\usepackage{amsfonts}
\usepackage{graphicx}
\usepackage{enumitem}
\usepackage{pgf,tikz,pgfplots}
\pgfplotsset{compat=1.15}
\usepackage{mathrsfs}

\usetikzlibrary{arrows}
\newtheorem{defn}{Definition}[section]
\newtheorem{theorem}{Theorem}[section]
\newtheorem{corollary}[theorem]{Corollary}
\newtheorem{proposition}[theorem]{Proposition}

\newenvironment{proof}{\smallskip\noindent{{\it Proof.}}\hskip \labelsep}%
           {\hfill\penalty10000\raisebox{-.09em}{\large\bf\rm $\blacksquare$}\par\medskip}

\newtheorem{lemma}[theorem]{Lemma}

\newtheorem{example}{Example}[section]

\begin{document}
\begin{frontmatter}

\title{Explicit multivariate approximations from cell-average data}\tnotetext[label1]{The first and third author have been supported through project 20928/PI/18 (Proyecto financiado por la Comunidad Aut\'onoma de la Regi\'on de Murcia a trav\'es de la convocatoria de Ayudas a proyectos para el desarrollo de investigaci\'on cient\'ifica y t\'ecnica por grupos competitivos, incluida en el Programa Regional de Fomento de la Investigaci\'on Cient\'ifica y T\'ecnica (Plan de Actuaci\'on 2018) de la Fundaci\'on S\'eneca-Agencia de Ciencia y Tecnolog\'ia de la Regi\'on de Murcia) and by the national research project PID2019-108336GB-I00. The fourth author
has been supported by grant MTM2017-83942 funded by Spanish MINECO and by grant PID2020-117211GB-I00 funded by MCIN/AEI/10.13039/501100011033.}

\author[UPCT]{Sergio Amat}
\ead{sergio.amat@upct.es}
\author[TEL]{David Levin}
\ead{levindd@gmail.com}
\author[UPCT]{Juan Ruiz}
\ead{juan.ruiz@upct.es}
\author[UV]{Dionisio F. Y\'a\~nez}
\ead{dionisio.yanez@uv.es}
\date{Received: date / Accepted: date}

\address[UPCT]{Departamento de Matemática Aplicada y Estadística. Universidad  Polit\'ecnica de Cartagena. Cartagena, Spain.}
\address[TEL]{School of Mathematical Sciences. Tel-Aviv University. Tel-Aviv, Israel.}
\address[UV]{Departamento de Matem\'aticas, Facultad de Matemáticas. Universidad de Valencia. Valencia, Spain.}

\begin{abstract}
Given gridded cell-average data of a smooth multivariate function, we present a constructive explicit procedure for generating a high-order global approximation of the function. One contribution is the derivation of high order approximations to point-values of the function directly from the cell-average data. The second contribution is the development of univariate B-spline based high order quasi-interpolation operators using cell-average data.
Multivariate spline quasi-interpolation approximation operators are obtained by tensor products of the univariate operators.
\end{abstract}


\end{frontmatter}

\section{Introduction}

In some applications, e.g., in some numerical methods for solving flow problems, and in medical imaging, our raw data is in the form of cell-averages on a grid of cells. The underlying function is usually piecewise smooth, and the ultimate challenge is to find a high order piecewise smooth approximation to the function, including a high order approximation of the boundaries between the smooth pieces. This is the subject of several papers \cite{AD,SHU_cell}, including a recent paper by the authors \cite{IMAJNA}. A basic ingredient of these works is the simpler problem of approximating a smooth function from cell-average data,
and this is the subject of the present contribution.

Given the cell averages of a smooth function, we derive high order approximations of the values of the function at the middle point of each cell using the cell-average plus a linear combination of local even order central differences of the cell-average data. We start with the univariate case. Let $f$ be a $\mathcal{C}^{2m+2}$ function, and consider the cell-averages of $f$ on intervals of length $h$:
\begin{equation}\label{eq01}
\bar f(a)=\frac{1}{h}\int_{a-h/2}^{a+h/2}f(x)dx.
\end{equation}
Under these premises, we show that
\begin{equation}\label{eq02}
f(a)=\bar f(a) +\sum_{r=1}^ma_r \Delta^{2r}\bar f(a)+O(h^{2m+2}),
\end{equation}
where $\Delta^2g(a)=g(a-h)-2g(a)+g(a+h)$, and we develop a recursive algorithm for computing the coefficients $\{a_r\}$.
Next, we extend the above 1D formula to 2D and, by induction, to $k$D. We also show that the $k$D formula is just a tensor product of the 1D formula.

One way of computing a high order approximation of $f$ using its cell-averages is by finding a local polynomial whose cell-averages agree with the given cell-average data. This approach is used e.g. in \cite{ACDD}. Here we aim at getting a global approximation of $f$, and for this purpose we suggest using the quasi-interpolation strategy.

The term quasi-interpolation approximation describes a local high order approximation which is of the same approximation order as an interpolation operator using the same approximation function space. Spline quasi-interpolation operators are well studied, and explicit schemes are developed in \cite{speleers}. The $p$-order quasi-interpolation operators in \cite {speleers} compute the spline approximation of $f(x)$ using a small number of equidistant function values near $x$, such that the approximation is exact if $f\in\Pi^p(\mathbb{R})$, the space of polynomials of degree $\le p$. In the present paper we derive quasi-interpolation operators which use cell-average data. The new quasi-interpolation schemes are local and reproduce $\Pi^p(\mathbb{R})$, and we show that they achieve the full approximation order, i.e., $O(h^{p+1})$.

The derivation for the univariate case is based upon Speleers formulae \cite{speleers}  together with some basic properties of B-splines. The application to higher dimensions can be easily done by applying the tensor product of the univariate operators.

In practice, the cell-average data is given on a bounded rectangular domain. Thus, the algorithms for approximating the point-values of $f$ and for computing the spline approximation need to be changed near the boundaries of the domain. We suggest here a way of treating the approximation near the boundaries, maintaining the same asymptotic approximation error rates, only with larger multiplying constants.
This paper is divided in three sections: Firstly, we present a direct procedure in order to obtain point-wise values of the function from the cell averages, we show the error formula and extend the results to $k$D dimensions. In the second section, quasi-interpolation using B-splines is explained and a new general method to obtain a global approximation to the function is implemented. Also, some specific examples are presented and the leading error coefficient is derived. Finally, some conclusions and future work are commented in section 4.

\section{From cell-average to point value from a Lagrange interpolation point of view}\label{sec1}
Typically, in order to recover the original value of a function in a determined point from cell-average data we solve a system, i.e. if we consider the data $\{\bar{f}(a+jh)\}_{j=-m}^m$ and we want to approximate $f(a)$, we pose the system with unknowns $c_j, \,j=0,\hdots,2m+1$:
\begin{equation}\label{sistema}
\int_{a+jh-\frac{h}{2}}^{a+jh-\frac{h}{2}}\left(\sum_{i=0}^{2m+1} c_i x^i\right)dx=\bar{f}(a+jh), \quad j=-m,\hdots,m,
\end{equation}
define the polynomial $g\in\Pi^{2m+1}_1(\mathbb{R})$ as
$$g(x):=\sum_{i=0}^{2m+1} c_i x^i,$$
and calculate an approximation of $f(a)$ evaluating $g$ at $a$. It is well-known that, if $f\in \mathcal{C}^{2m+2}$, with this procedure we obtain an approximation of order $O(h^{2m+2})$. In this section, we develop a direct method  to obtain a general centered expression for a determined order of accuracy which relates the values of a function at a point with the cell-average around this point, Equation \eqref{eq01}.
The idea is to use the Taylor expansion of $f$ in Equation \eqref{eq01}, replacing the derivatives of $f$ by its appropriate finite differences, and then expressing the finite differences of $f$ by finite differences of $\bar f$. This leads to a triangular system for the coefficients in (2), which allows to design a recursive algorithm to get the solution.
This section is divided into two parts, we develop the method in the 1D case and provide an error formula, then we extend the result to the 2D case via the tensor-product strategy and finally, via an induction process, to $k$D.

\subsection{A procedure to obtain a centered expression for any order of accuracy}\label{generalcentro}
Along all this section, we will suppose a point $a\in \mathbb{R}$, assume $f\in \mathcal{C}^{2m+2}(\Omega)$, $\Omega=[x_1,x_2]$, $h>0$, and
$[a-m h,a+m h]\subset \Omega $. We denote,
\begin{equation}\label{qesdelta}
\Delta^{2m} \bar{f}(a)=\sum_{j=-m}^m (-1)^{m+j} \binom{2m}{j+m}\bar{f}(a+jh).
\end{equation}
To ease the notation, we express the calculations at the point $a=0$. Using Taylor expansions on Equation \eqref{eq01}, there exists $\xi_1 \in \Omega$ such that
\begin{equation}\label{eqd1}
\bar f(0)=\sum_{t=0}^m \frac{1}{(2t+1)!2^{2t}}h^{2t}f^{(2t)}(0)+\frac{1}{(2m+3)!2^{2m+2}}h^{2m+2}f^{(2m+2)}(\xi_1).
\end{equation}
We denote,
\begin{equation}\label{eqbb}
b_{m+1}=-\frac{1}{(2m+3)!2^{2m+2}}.
\end{equation}
Let $g\in\mathcal{C}^{2m+2}(\Omega)$ be a function with $[-rh,rh]\subset\Omega$ and $r\leq m$, using  Taylor expansions again we get:
\begin{equation}\label{eqd2}
g((j-r)h)+g(-(j-r)h)=2g(0)+2\sum_{k=1}^m \frac{(h(j-r))^{2k}}{(2k)!}g^{(2k)}(0)+O(h^{2m+2}),
\end{equation}
and
\begin{equation}\label{eqd3}
\Delta^{2r} g(0)=\binom{2r}{r}(-1)^r g(0)+\sum_{j=r+1}^{2r} (-1)^j\binom{2r}{j}\left(g((j-r)h)+g(-(j-r)h)\right).
\end{equation}
In order to obtain an easier expression for $\Delta^{2r} g(0)$ we  prove the following well-known auxiliary lemma.
\begin{lemma}\label{lema1}
Let $r$ be a natural number with $r\geq 1$ then
$$  \binom{2r}{r}(-1)^r+2\sum_{j=r+1}^{2r} (-1)^j\binom{2r}{j}=0.$$
\end{lemma}
\begin{proof}
From
\begin{equation*}
\binom{2r}{j}=\binom{2r}{2r-j}, \quad 0\leq j\leq 2r,
\end{equation*}
and the Newton binomial expansion:
$$w(x):=(x-1)^{2r}=\sum_{j=0}^{2r}\binom{2r}{j} x^{2r-j}(-1)^j,$$
we have that
$$0=w(1)=\sum_{j=0}^{2r}\binom{2r}{j} (-1)^j = \binom{2r}{r}(-1)^r+2\sum_{j=r+1}^{2r} (-1)^j\binom{2r}{j}.$$
\end{proof}
Thus, using \eqref{eqd2}, \eqref{eqd3} and Lemma \ref{lema1} we get:
\begin{equation}\label{eqd0}
\begin{split}
\Delta^{2r} g(0)=&\sum_{j=r+1}^{2r} (-1)^j\binom{2r}{j}\left(2\sum_{k=1}^m \frac{(h(j-r))^{2k}}{(2k)!}g^{(2k)}(0)+O(h^{2m+2})\right) \\
=& 2\sum_{k=1}^m\left(\sum_{j=r+1}^{2r} (-1)^j\binom{2r}{j}(j-r)^{2k}\right)\frac{h^{2k}}{(2k)!}g^{(2k)}(0)+O(h^{2m+2}),
\end{split}
\end{equation}
and prove the following direct consequence, Corollary \ref{cor2}.
\begin{corollary}\label{cor2}
Let $r,k$ be natural numbers with $1\leq k \leq r$ then
$$\sum_{j=r+1}^{2r} (-1)^j\binom{2r}{j}(j-r)^{2k}=\begin{cases}
0, & 1\leq k < r,\\[2pt]
\frac{(2r)!}{2} & k=r.
\end{cases}
$$
\end{corollary}
\begin{proof}
Let $r\geq1$ be a natural number. We take a function, $g\in\mathcal{C}^{2r+1}(\mathbb{R})$ such that $g^{(2k)}(0)\neq 0, \,\, 1\leq k \leq r$ and $g^{(2r+1)}(0)=0$ (for example $g(x)=(x+1)^{2r}$), then we define the polynomial of degree $2r$:
$$\Delta^{2r}g(0)=h^{2r} g^{(2r)}(0)=:p(h),$$
and by Equation \eqref{eqd0} we obtain the following polynomial:
$$\Delta^{2r}g(0)=2\sum_{k=1}^r\left(\sum_{j=r+1}^{2r} (-1)^j\binom{2r}{j}(j-r)^{2k}\right)\frac{h^{2k}}{(2k)!}g^{(2k)}(0)=:q(h),$$
from $q(h)=p(h)$ we get the result.
\end{proof}

Using  Corollary \ref{cor2}, we rewrite Equation \eqref{eqd0} as:
\begin{equation}\label{eqd4}
\Delta^{2r} g(0)= 2\sum_{k=r}^m\left(\sum_{j=r+1}^{2r} (-1)^j\binom{2r}{j}(j-r)^{2k}\right)\frac{h^{2k}}{(2k)!}g^{(2k)}(0)+O(h^{2m+2}).
\end{equation}
Finally, if we substitute $g=f^{(2t)}$, with $t\leq m$ in Equation \eqref{eqd4} then for each $t$
\begin{equation}\label{eqd5}
\Delta^{2r} f^{(2t)}(0)= 2\sum_{k=r}^{m-t}\left(\sum_{j=r+1}^{2r} (-1)^j\binom{2r}{j}(j-r)^{2k}\right)\frac{h^{2k}}{(2k)!}f^{(2(k+t))}(0)+O(h^{2m+2}).
\end{equation}
Applying $\Delta^{2r}$ to Equation \eqref{eqd1},
\begin{equation*}
\begin{split}
\Delta^{2r} \bar f(0)=&\sum_{t=0}^m \frac{1}{(2t+1)!2^{2t}}h^{2t}f^{(2t)}(0)+b_{m+1}h^{2m+2}f^{(2m+2)}(\xi_1),\\
=&\sum_{t=0}^m \frac{h^{2t}}{(2t+1)!2^{2t}}\left(2\sum_{k=r}^{m-t}\left(\sum_{j=r+1}^{2r} (-1)^j\binom{2r}{j}(j-r)^{2k}\right)\frac{h^{2k}}{(2k)!}f^{(2(k+t))}(0)+O(h^{2m+2})\right)\\
&-b_{m+1}h^{2m+2}f^{(2m+2)}(\xi_1)\\
=&\sum_{t=0}^m \left(\sum_{k=r}^{m-t}\left(\sum_{j=r+1}^{2r} (-1)^j\binom{2r}{j}(j-r)^{2k}\right)\frac{h^{2(k+t)}}{(2t+1)!(2k)!2^{2t-1}}f^{(2(k+t))}(0)\right)+O(h^{2m+2}).\\
\end{split}
\end{equation*}
Defining $i=k+t$ for $r\leq k\leq m-t$ and $0\leq t$, we get
\begin{equation}\label{eqd6}
\begin{split}
\Delta^{2r} \bar f(0)=&\sum_{t=0}^m \left(\sum_{k=r}^{m-t}\left(\sum_{j=r+1}^{2r} (-1)^j\binom{2r}{j}(j-r)^{2k}\right)\frac{h^{2(k+t)}}{(2t+1)!(2k)!2^{2t-1}}f^{(2(k+t))}(0)\right)+O(h^{2m+2})\\
=&\sum_{i=r}^m\left( \sum_{t=0}^{i-r}\left(\frac{1}{2^{2t-1}(2t+1)!(2(i-t))!}\sum_{j=r+1}^{2r} (-1)^j\binom{2r}{j}(j-r)^{2(i-t)}\right)h^{2i}f^{(2i)}(0)\right)+O(h^{2m+2}).\\
\end{split}
\end{equation}
We introduce the notation
\begin{equation}\label{eqd7}
\Delta^{2r}_{m} \bar f(0)=\sum_{i=r}^{m}\left( \sum_{t=0}^{i-r}\left(\frac{1}{2^{2t-1}(2t+1)!(2(i-t))!}\sum_{j=r+1}^{2r} (-1)^j\binom{2r}{j}(j-r)^{2(i-t)}\right)h^{2i}f^{(2i)}(0)\right).
\end{equation}
From \eqref{eqd1}, we have that
\begin{equation*}
\bar f(0)=f(0)+\sum_{t=1}^{m} \frac{h^{2t}}{(2t+1)!2^{2t}}f^{(2t)}(0)+O(h^{2m+2}).
\end{equation*}
Now we look for coefficients $\mathbf{a}_{m}=(a_1,\hdots,a_{m})$ such that
\begin{equation}\label{levink}
\sum_{t=1}^{m} \frac{h^{2t}}{(2t+1)!2^{2t}}f^{(2t)}(0)+\sum_{r=1}^{m} a_r \Delta^{2r}_{m} \bar f(0)=0.
\end{equation}
In view of the expression \eqref{eqd7} we obtain:
\begin{equation}\label{sys}
\mathbf{M}_{m} \mathbf{a}_{m}^T = \mathbf{b}_{m}^T,
\end{equation}
where 
$\mathbf{b}_{m}=(b_1,\hdots,b_{m})$ are the coefficients of the derivatives which appear in the expression of $\bar f(0)$, Equation \eqref{eqbb}, i.e.
\begin{equation}\label{valoresb}
b_i=-\frac{1}{(2i+1)!2^{2i}}, \quad 1\leq i \leq m,
\end{equation}
and $\mathbf{M}_{m}=(m_{i,j})_{i,j=1}^{m}$ are the coefficients of the derivatives of $\Delta^{2r}_{m}$, Equation \eqref{eqd7}, with
\begin{equation}\label{valoresmatriz}
m_{i,j}= \sum_{t=0}^{i-j}\frac{1}{2^{2t-1}(2t+1)!(2(i-t))!}\sum_{s=j+1}^{2j} (-1)^s\binom{2j}{s}(s-j)^{2(i-t)}.
\end{equation}

The following result is a direct consequence by Corollary \ref{cor2}. This result confirms that there exists a unique solution to our problem. Combining \eqref{eqd1} and \eqref{levink}, it follows that we can approximate the value of the function with a certain order of accuracy using a combination of cell-average data.
\begin{proposition}
Let $m$ be a natural number and $\mathbf{M}_{m}=(m_{i,j})_{i,j=1}^{m}$ the matrix defined in Equation \eqref{valoresmatriz}
then
$$m_{i,j}=\begin{cases} 0, & i<j, \\
1, & i=j.
\end{cases}$$
Therefore, the system defined in Equation \eqref{sys} has a unique solution.
\end{proposition}
In order to clarify all the concepts, we introduce an example.
\begin{example}
For example, for $m=5$, we get:
\begin{equation}\label{variablessystem}
\bordermatrix{
 ~ & {\Delta_5^2\bar{f}(a)}  & {\Delta_5^4\bar{f}(a)} & {\Delta_5^6\bar{f}(a)} & {\Delta_5^8\bar{f}(a)} & {\Delta^{10}_5\bar{f}(a)} \cr
h^2f^{(2)}(a)&  1 & 0 & 0 & 0 & 0 \cr
h^4f^{(4)}(a)& \frac{1}{8} & 1 & 0 & 0 & 0 \cr
h^6f^{(6)}(a)& \frac{13}{1920} & \frac{5}{24} & 1 & 0 & 0 \cr
h^8f^{(8)}(a)& \frac{41}{193536} & \frac{23}{1152} & \frac{7}{24} & 1 & 0 \cr
h^{10}f^{(10)}(a)& \frac{671}{154828800} & \frac{227}{193536} & \frac{77}{1920} & \frac{3}{8} & 1 \cr
}=\mathbf{M}_5,
\end{equation}
\begin{equation}\label{variablessystem2}
\bordermatrix{
 ~ & {f(a)-\bar{f}(a)}\cr
h^2f^{(2)}(a)& -(24)^{-1}\cr
h^4f^{(4)}(a)& -(1920)^{-1}\cr
h^6f^{(6)}(a)& -(322560)^{-1} \cr
h^8f^{(8)}(a)& -(92897280)^{-1}\cr
h^{10}f^{(10)}(a)& -(40874803200)^{-1}
}=\mathbf{b}_5^T.
\end{equation}
It is clear that there exists a unique solution and it is:
\begin{equation*}
 \mathbf{a}_5=\left(-\frac{1}{24},\frac{3}{640},-\frac{5}{7168},\frac{35}{294912},-\frac{63}{2883584}\right).
\end{equation*}
Then,
\begin{equation*}
\bar f(a)-\frac{1}{24}\Delta^2\bar f(a)+\frac{3}{640}\Delta^4\bar f(a)-\frac{5}{7168}\Delta^6\bar f(a)+\frac{35}{294912}\Delta^8\bar f(a)-\frac{63}{2883584}\Delta^{10}\bar f(a)=f(a)+O(h^{12}).
\end{equation*}
\end{example}
If we know the coefficients $a_1, \hdots,a_{m-1}$ we can write that
$$ a_{m}=b_{m}-\sum_{r=1}^{m-1} m_{m,r}a_r,$$
then we can design a recursive algorithm to solve the problem.
Collecting all the above results, we have proved that if $f\in \mathcal{C}^{2m+2}(\Omega)$ then
\begin{equation}\label{importante}
f(a)=\sum_{r=0}^{m} a_r \Delta^{2r} \bar f(a)+O(h^{2m+2}),
\end{equation}
being $a_0=1$ and $\mathbf{a}_{m}=(a_1,\hdots,a_{m})$ the solution of the system in \eqref{sys}. We present an error formula for this approximation in the following proposition.

\begin{proposition}[Error formula]\label{corolarioerror}
Let $\Omega$ be a closed interval, $1\leq m\in\mathbb{N}$; $a,h\in\mathbb{R}$, $h>0$ with $[a-mh,a+mh]\subset \Omega$, $f\in\mathcal{C}^{2m+2}(\Omega)$ and $||f^{(2m+2)}||_{\infty,\Omega}=\max\{|f^{(2m+2)}(x)|:x\in\Omega\}$, then
\begin{equation}\label{erroreq}
\left|f(a)-\sum_{r=0}^{m} a_r \Delta^{2r} \bar f(a)\right|\leq |a_{m+1}|||f^{(2m+2)}||_{\infty,\Omega}h^{2m+2},
\end{equation}
being $a_0=1$ and $\mathbf{a}_{m}=(a_i)_{i=1}^{m}$ the solution of the system \eqref{sys}, and
$$a_{m+1}=b_{m+1}-\sum_{r=1}^{m} m_{m+1,r}a_r,$$
with $b_{m+1}$ and $m_{m+1,j}$, $j=1,\hdots,m$ defined in Equations \eqref{valoresb} and \eqref{valoresmatriz} respectively.
\end{proposition}
\begin{proof}
For simplicity, we develop the proof for $a=0$. Using Taylor's expansions, we have that there exists $\xi_1\in[-\frac{h}{2},\frac{h}{2}]$ such that
\begin{equation}\label{eccor1}
\begin{split}
\bar f(0)&=f(0)+\sum_{t=1}^{m} \frac{h^{2t}}{(2t+1)!2^{2t}}f^{(2t)}(0)+\frac{1}{(2m+3)!2^{2m+2}}f^{(2m+2)}(\xi_1)h^{2m+2}\\
&=f(0)+\sum_{t=1}^{m} \frac{h^{2t}}{(2t+1)!2^{2t}}f^{(2t)}(0)-b_{m+1}f^{(2m+2)}(\xi_1)h^{2m+2}\\
&\leq f(0)+\sum_{t=1}^{m} \frac{h^{2t}}{(2t+1)!2^{2t}}f^{(2t)}(0)-b_{m+1}||f^{(2m+2)}||_{\infty,\Omega}h^{2m+2}
\end{split}
\end{equation}
with $b_{m+1}$ defined in Equation \eqref{valoresb}. Again, using Taylor's expansions, there exists $\xi_2\in[-mh,mh]$ such that
\begin{equation}\label{eccor15}
\begin{split}
\Delta^{2r} \bar f(0)=&\sum_{i=r}^{m}\left( \sum_{t=0}^{i-1}\left(\frac{1}{2^{2t-1}(2t+1)!(2(i-t))!}\sum_{j=r+1}^{2r} (-1)^j\binom{2r}{j}(j-r)^{2(i-t)}\right)h^{2i}f^{(2i)}(0)\right)+\\
&+\sum_{t=0}^{m}\frac{1}{2^{2t-1}(2t+1)!(2(m+1-t))!}\sum_{j=r+1}^{2r} (-1)^j\binom{2r}{j}(j-r)^{2(m+1-t)}h^{2m+2}f^{(2m+2)}(\xi_2)\\
=&\Delta_{m}^{2r} \bar f(0)+m_{m+1,r}f^{(2m+2)}(\xi_2)h^{2m+2}\\
\leq&\Delta_{m}^{2r} \bar f(0)+m_{m+1,r}||f^{(2m+2)}||_{\infty,\Omega}h^{2m+2}\\
\end{split}
\end{equation}
with $m_{m+1,r}$, $r=1,\hdots,m$ defined in Equation \eqref{valoresmatriz}. Finally, by the construction of the system \eqref{sys} we know:
\begin{equation}\label{eccor2}
\sum_{t=1}^{m} \frac{h^{2t}}{(2t+1)!2^{2t}}f^{(2t)}(0)+\sum_{r=1}^{m} a_r \Delta^{2r}_{m} \bar f(0)=0.
\end{equation}
Therefore, collecting Equations \eqref{eccor1}, \eqref{eccor15} and \eqref{eccor2}:
\begin{equation*}
\begin{split}
\left|\sum_{r=0}^{m} a_r \Delta^{2r} \bar f(0)-f(0)\right|&\leq\left|\bar{f}(0)-f(0)-\sum_{t=1}^{m} \frac{h^{2t}}{(2t+1)!2^{2t}}f^{(2t)}(0)+\sum_{r=1}^{m}a_r m_{m+1,r}||f^{(2m+2)}||_{\infty,\Omega}h^{2m+2}\right|\\
&\leq\left|-b_{m+1}||f^{(2m+2)}||_{\infty,\Omega}h^{2m+2}+\sum_{r=1}^{m}a_r m_{m+1,r}||f^{(2m+2)}||_{\infty,\Omega}h^{2m+2}\right|\\
&=\left|-b_{m+1}+\sum_{r=1}^{m}a_r m_{m+1,r} \right|||f^{(2m+2)}||_{\infty,\Omega}h^{2m+2}\\
&=|a_{m+1}|||f^{(2m+2)}||_{\infty,\Omega}h^{2m+2}.\\
\end{split}
\end{equation*}
\end{proof}

\begin{corollary}\label{papolinomios}
Let $m\in\mathbb{N}$ and $P\in\Pi^{2m+1}_1(\mathbb{R})$, then
$$\sum_{r=0}^{m} a_r \Delta^{2r}\bar P(a)=P(a),\,\,\, \forall  \,a\in\mathbb{R},$$
where $a_0=1$ and $\mathbf{a}_{m}=(a_i)_{i=1}^{m}$ is the solution of the system in \eqref{sys}.
\end{corollary}
\begin{proof}
Let $m$ be a natural number and $P(x)\in \Pi^{2m+1}_1(\mathbb{R})$. Then the result is a direct consequence of Proposition \ref{corolarioerror}, since
$P^{(2m+2)}(x)=0$ for all $x\,\in\mathbb{R}$.
\end{proof}

Finally, we show that this new recursive method is equivalent to solving the system introduced in Equation \eqref{sistema}. Note that with our algorithm it is not necessary to compute any coefficient of the polynomials. As a consequence, we can obtain a high order approximation with a low computational cost.

\begin{proposition}
Let $\Omega$ be a closed interval, $m\in\mathbb{N}$, $f\in \mathcal{C}^{2m+2}(\Omega)$, we consider $g\in\Pi^{2m+1}(\mathbb{R})$ such that
\begin{equation}\label{system}
\bar{g}(a+jh)=\bar{f}(a+jh), \quad j=-m,\hdots,m,
\end{equation}
then:
\begin{equation}
\sum_{r=0}^{m} a_r \Delta^{2r} \bar f(a)=g(a),
\end{equation}
where $a_0=1$ and $\mathbf{a}_{m}=(a_1,\hdots,a_{m})$ is the solution of the system \eqref{sys}.
\end{proposition}
\begin{proof}
By hypothesis, \eqref{system} and by Equation \eqref{qesdelta}, if $r\leq m$ we have
\begin{equation*}
\Delta^{2r} \bar{f}(a)=\sum_{j=-r}^r (-1)^{r+j} \binom{2r}{j+r}\bar{f}(a+jh)=\sum_{j=-r}^r (-1)^{r+j} \binom{2r}{j+r}\bar{g}(a+jh),
\end{equation*}
then by Corollary \ref{papolinomios}, we get
$$\sum_{r=0}^{m} a_r \Delta^{2r} \bar f(a)=\sum_{r=0}^{m} a_r \Delta^{2r} \bar g(a)=g(a).$$
\end{proof}
As a conclusion, we have developed a direct method to obtain an approximation to the values of the function $f$ from cell-average data. We use a tensor-product strategy to extend the result to 2D; see \ref{seccion2d} for more information. With this extension, it is easy to generalise to $k$ dimensions.

\subsection{From cell-average to point values: The $k$D case}

The generalization to $k$D case is easy using an induction process. Therefore, let $(s_1,\hdots,s_k)\in\mathbb{R}^k$ be an arbitrary point, $\Omega \subset \mathbb{R}^k$ be an open set, with
$\tilde\Omega:=[s_1-m^{x_1} h_{x_1},s_1+m^{x_1} h_{x_1}]\times \hdots \times [s_k-m^{x_k} h_{x_k},s_k+m^{x_k} h_{x_k}]\subset \Omega$ and $f\in \mathcal{C}^p(\Omega)$, $p\geq 2m^{x_j}+2$, $j=1,\hdots,k$. We define:
$$\bar{f}(s_1,\hdots,s_k)=\int_{\tilde \Omega}f(x_1,\hdots,x_k)dx_1\hdots dx_k,$$
then using the same process as in
\ref{seccion2d} we have:
\begin{equation}\label{ecuacioncasond1}
\begin{split}
f(s_1,\hdots,s_k)&=\sum_{i_1=0}^{m^{x_1}}\sum_{i_2=0}^{m^{x_2}}\hdots \sum_{i_k=0}^{m^{x_k}}a_{i_1}a_{i_2}\hdots a_{i_n}\Delta^{2i_1,x_1}\Delta^{2i_2,x_2}\hdots \Delta^{2i_k,x_k}\bar{f}(s_1,\hdots,s_k)+\sum_{j=1}^kO(h_{x_j}^{2m^{x_j}_0+2}).
\end{split}
\end{equation}

Finally, following the same steps as in the previous section, we show a corollary that indicates that this method is exact for a certain class of polynomials. Firstly, we introduce the notation of the tensor product of polynomials. Letting $\mathbf{p}=(p_1,\hdots,p_k)\in\mathbb{N}^k$, we can define
\begin{equation}\label{polinomiosvarias}
\Pi_k^{\mathbf{p}}(\mathbb{R})=\left\{P(x_1,\hdots,x_k)=\sum_{j_1=0}^{p_1}\hdots\sum_{j_k=0}^{p_k}a_{j_1\hdots j_k}x_1^{j_1}\hdots x_k^{j_k}:a_{j_1\hdots j_k}\in\mathbb{R}, 0\leq j_l\leq p_l, l=1,\hdots,k\right\}.
\end{equation}

\begin{corollary}
Let $m^{x_j}\in\mathbb{N}$, $j=1,\hdots,k$ and $P\in\Pi^{(2m^{x_1}+1,2m^{x_2}+1,\hdots,2m^{x_k}+1)}_k(\mathbb{R})$, then
$$P(s_1,\hdots,s_k)=\sum_{i_1=0}^{m^{x_1}}\sum_{i_2=0}^{m^{x_2}}\hdots \sum_{i_k=0}^{m^{x_k}}a_{i_1}a_{i_2}\hdots a_{i_k}\Delta^{2i_1,x_1}\Delta^{2i_2,x_2}\hdots \Delta^{2i_k,x_k}\bar{P}(s_1,\hdots,s_k),\,\,\, \forall  \,(s_1,\hdots,s_k)\in\mathbb{R}^k,$$
where $a_0=1$ and $\mathbf{a}_{m}=(a_i)_{i=1}^{m}$ is the solution of the system \eqref{sys}.
\end{corollary}

\begin{example}
We consider $n=3$, $0<h=h_{x}=h_{y}=h_{z}$ and $m^{x}=m^{y}=m^{z}=1$, then
\begin{equation*}
\begin{split}
f(a,b,c)=&\bar{f}(a,b,c)-\frac{1}{24}(\Delta^{2,x}\bar{f}(a,b,c)+\Delta^{2,y}\bar{f}(a,b,c)+\Delta^{2,z}\bar{f}(a,b,c))+\frac{1}{576}\Delta^{2,x}\Delta^{2,y}\bar{f}(a,b,c)\\
&+\frac{1}{576}\left(\Delta^{2,x}\Delta^{2,z}\bar{f}(a,b,c)+\Delta^{2,y}\Delta^{2,z}\bar{f}(a,b,c)\right)-\frac{1}{13824}\Delta^{2,x}\Delta^{2,y}\Delta^{2,z}\bar{f}(a,b,c)+O(h^{4}).\\
\end{split}
\end{equation*}
\end{example}

\section{Explicit full order global spline approximation over a rectangle}

\subsection{Univariate quasi-interpolation operators from $q$-average data}\label{univariate}

Given the values of  $f\in \mathcal{C}^{p+1}(\mathbb{R})$ at equidistant points $x_n=nh$, $n \in\mathbb{Z}$, an efficient way of achieving a global $O(h^{p+1})$ approximation of $f$ is by simple local combinations of B-splines of degree $p$ (see e.g. \cite{speleers}). For example, let $\{f(nh)\}$ be the values of $f\in \mathcal{C}^4(\mathbb{R})$, and let $B_3(x)$ be the cubic B-spline supported on $[-2,2]$, with equidistant knots $\{-2,-1,0,1,2\}$. Defining the quasi-interpolation operator $Q_3$ as
\begin{equation}\label{quasi}
Q_3(f)(x)=\sum_{n\in\mathbb{Z}}\big(-\frac{1}{6}f((n-1)h)+\frac{4}{3}f(nh)-\frac{1}{6}f((n+1)h)\big)B_3(\frac{x}{h}-n),
\end{equation}
it follows that $Q_3$ reproduces $\Pi^3(\mathbb{R})$, the space of polynomials of degree $\le 3$, and on any finite interval $I$,
$$\|f-Q_3(f)\|_\infty=O(h^4),$$
as $h\to 0$.

As we have mentioned in many cases, the given data is the cell-average of an integrable function or another local average (see e.g. \cite{donohoetal,harten}), as the hat-average. Thus, it is not possible to use the quasi-interpolation formula, for example, the one shown in Equation \eqref{quasi}. An alternative possibility is to define this operator depending on the new type of data using cell-averages. In this case, the reproduction of polynomials is not ensured. In this section, we prove a relation between the point-value and cell-average operators depending on the degree of the B-spline used, and we generalise this relation. This section is divided in two parts: In the first we introduce the problem and prove some auxiliary lemmas. In the second part we show the main theorem and introduce some examples.

\subsubsection{Notation: local averages and auxiliary results}

We denote by $B_p(x)$ the $p$-degree B-spline supported on
\begin{equation}\label{Ip}
I_p =\left[-\frac{p+1}{2},\frac{p+1}{2}\right],
\end{equation}
with equidistant knots
\begin{equation}\label{knots}
S_p = \left\{-\frac{p+1}{2},\hdots,\frac{p+1}{2}\right\}.
\end{equation}
Let $\{f(nh)\}$ be the values of $f$. We define the vector
$$f_{n,p}=(f_{n-\left\lfloor \frac{p}{2} \right\rfloor},\hdots,f_{n+\left\lfloor \frac{p}{2} \right\rfloor}),$$
where $f_n=f(nh)$ and the floor function $\lfloor \cdot \rfloor:\mathbb{R}\to \mathbb{Z}$ and ceiling function $\lceil \cdot \rceil:\mathbb{R}\to \mathbb{Z}$ are:
$$\lfloor x \rfloor=\max\{z\in\mathbb{Z}:z\leq x\}, \quad \lceil x \rceil=\min\{z\in\mathbb{Z}:z\geq x\}.$$
The classical quasi-interpolation operator $Q_p$ is defined as
\begin{equation}\label{operatorQ}
Q_p(f)(x)=\sum_{n\in \mathbb{Z}}L_{p}(f_{n,p})B_p\left(\frac{x}{h}-n \right),
\end{equation}
with the linear operator $L_p:\mathbb{R}^{2\left\lfloor \frac{p}{2} \right\rfloor+1}\to \mathbb{R}$ defined using the ideas presented in \cite{speleers} as:
\begin{equation}\label{operadorL}
L_p(f_{n,p})=\sum_{j=-\left\lfloor \frac{p}{2} \right\rfloor}^{\left\lfloor \frac{p}{2} \right\rfloor} c_{p,j} f_{n+j},
\end{equation}
where the coefficients $c_{p,j}$, $j=-\left\lfloor \frac{p}{2} \right\rfloor,\hdots,\left\lfloor \frac{p}{2} \right\rfloor$ are:
\begin{equation}\label{coeficientesc}
c_{p,j}=\sum_{l=0}^{\left\lfloor \frac{p+1}{2} \right\rfloor-1}\frac{t(2l+p+1,p+1)}{{{2l+p+1}\choose{p+1}}}\sum_{i=0}^{2l}\frac{(-1)^i}{i!(2l-i)!}\delta_{l-i+\left\lceil\frac{p+1}{2}\right\rceil,j+1+\left\lfloor\frac{p}{2}\right\rfloor},
\end{equation}
 $\delta_{i,j}$ the Kronecker delta, i.e.,
$$ \delta_{i,j}=\left\{
       \begin{array}{ll}
            1, &\hbox{if} \,\, i=j; \\
         0, &\hbox{if} \,\, i\neq j; \\
       \end{array}
     \right.
$$
and $t(i,j)$ are the central factorial numbers of the first kind (see \cite{speleers} and \cite{butzeretal}) which can be computed recursively as:
\begin{equation*}
t(i,j)=\left\{
         \begin{array}{ll}
           0, & \hbox{if}\,\, j>i, \\
           1, & \hbox{if}\,\, j=i, \\
           t(i-2,j-2)-\left(\frac{i-2}{2}\right)^2t(i-2,j), & \hbox{if} \,\, 2\leq j <i,
         \end{array}
       \right.
\end{equation*}
with
$$t(i,0)=0,\quad t(i,1)=\prod_{l=1}^{m-1}\left(\frac{m}{2}-l\right),\,\, i\geq 2,$$
and $t(0,0)=t(1,1)=1, t(0,1)=t(1,0)=0$.
We introduce the norm of the operator $L_p$ which will be useful to prove the approximation order of the new operators.
\begin{defn}\label{def:normaL}
The infinity norm of the operator $L_p$, Equation \eqref{operadorL}, is defined as:
$$||L_p||_\infty=\sum_{j=-\left\lfloor \frac{p}{2} \right\rfloor}^{\left\lfloor \frac{p}{2} \right\rfloor} |c_{p,j}|.$$
\end{defn}
In Table \ref{tablaCs} some values of $c_{p,j}$ are showed, jointly with the norm of  the operator $L_p$.
\begin{table}[htbp]
  \begin{center}
   \begin{tabular}{crrrrrrr}\hline
                             & $c_{1,j}$ & $c_{2,j}$ & $c_{3,j}$ & $c_{4,j}$& $c_{5,j}$ & $c_{6,j}$  &$c_{7,j}$     \\    \hline
                 $ j=0$      &  $1$  & $5/4 $    & $4/3$    & $319/192$  & $73/40$  & $79879/34560$  & $2452/945$             \\
                 $ j=1$      &       & $-1/8 $   & $-1/6$   & $-107/288$ & $-7/15$  & $-37003/46080$ & $-1657/1680$             \\
                 $ j=2$      &       &           &          & $47/1152$  & $13/240$ & $751/4608$     & $22/105$     \\
                 $ j=3$      &       &           &          &            &          & $-2159/138240$ & $-311/15120$\\\hline
$||L_p||_\infty$                    & $1$   & $3/2$     & $5/3$    & $179/72$   & $43/15$  & $9233/2160$    & $4751/945$ \\\hline
   \end{tabular}
    \caption{The values are symmetric, i.e. $c_{p,j}=c_{p,-j}$, $k=0,\hdots,\left\lfloor\frac{p}{2}\right\rfloor$}
    \label{tablaCs}
  \end{center}
\end{table}

In general, we have the following result about quasi-interpolation by splines with a uniform knots' sequence (see for example \cite{speleers}):
\begin{proposition}\label{quasiintprop}
Consider a $p$ degree B-spline, $B_{p}$ with knots $S_p$, Equation \eqref{knots}, then the local operator $Q_{p}$ defined in Equation \eqref{operatorQ}  reproduces $\Pi^{p}(\mathbb{R})$ and achieves $O(h^{p+1})$ approximation order.
\end{proposition}

In this section we develop an extended version of the operator $Q_p$ using the cell-average discretisation defined in Equation \eqref{eq01}:
\begin{equation*}
\bar f(nh)=\frac{1}{h}\int_{nh-\frac{h}{2}}^{nh+\frac{h}{2}} f(x) dx,
\end{equation*}
which is usually employed in image processing or signal processing applications. In general, using the notation introduced by Harten in \cite{harten}, we define for $q\in \mathbb{N}$ the weight functions $\omega^q$ as the repeated convolution:
\begin{equation}\label{omegadefinition}
\begin{split}
&\omega^0(x)=\delta_{x,0},\\
&\omega^q(x)=\omega^{q-1} \ast \chi_{[-\frac{1}{2},\frac{1}{2}]}(x), \quad q\in\mathbb{N},
\end{split}
\end{equation}
using the characteristic function in the interval $[-1/2,1/2]$:
\begin{equation*}
\chi_{[-\frac{1}{2},\frac{1}{2}]}=\left\{
         \begin{array}{ll}
           1, & \hbox{if}\,\, x\in [-\frac{1}{2},\frac{1}{2}], \\
           0, & \hbox{otherwise}.\\
         \end{array}
       \right.
\end{equation*}
Note that $\omega^{p+1}=B_p$. With respect to these even functions, given a value $h>0$, we consider the $q$-average data (or $q$-moment, see \cite{donohoetal}) as the inner products:
\begin{equation}\label{defdiscretizacion}
f^{q}_n:=\left\langle f, \frac{1}{h}\omega^q\left(\frac{\cdot}{h}-n \right)\right\rangle=\frac{1}{h}\int_{\mathbb{R}} f(x)\omega^q\left(\frac{x}{h}-n\right)dx =\frac{1}{h}\int_{\mathbb{R}} f(x)\omega^q\left(\frac{nh-x}{h}\right)dx = (f\ast \omega^q_h)(n h),
\end{equation}
with
$$\omega^q_h(x)=\frac{1}{h}\omega^q\left(\frac{x}{h}\right).$$
If $q=0$ the point-value discretisation is recovered,
\begin{equation*}
f^{0}_n=f(nh)=f_n,
\end{equation*}
and if $q=1$ we get the cell-average data
\begin{equation*}
f^{1}_n=\frac{1}{h}\int_{nh-\frac{h}{2}}^{nh+\frac{h}{2}} f(x) dx=\bar f(nh)=:\bar f_n.
\end{equation*}
In the following we construct a function $\Theta^q_{p}:\mathbb{R}^{2\xi_q(p)+1}\to \mathbb{R}$ and an associated operator denoted by:
\begin{equation}\label{operatorbarQ}
Q^q_p(f)(x)=\sum_{n\in \mathbb{Z}}\Theta^q_{p}(f^q_{n,\xi_q(p)})B_p\left(\frac{x}{h}-n \right),
\end{equation}
with
$$f^q_{n,\xi_q(p)}= \left(f^{q}_{n-\left\lfloor \frac{\xi_q(p)}{2} \right\rfloor},\hdots,f^{q}_{n+\left\lfloor \frac{\xi_q(p)}{2} \right\rfloor}\right),$$
where $\xi_p: \mathbb{N}\to \mathbb{N}$, and such that under some conditions
$$Q^q_p(P)=Q_p(P),$$
for any polynomial $P\in\Pi^p(\mathbb{R})$. Thus, we will have an operator based on cell-average data which approximates the function with the same order of accuracy as the approximation defined by \eqref{operatorQ} using point-values.

Let's start with some auxiliary results. We prove the following well-known result in order to use it in the proof of the main result.
\begin{lemma}\label{propconv}
Consider a $p$ degree B-spline, $B_{p}$ with knots $S_p$ (Equation \eqref{knots}) and the function $\omega^q$, with $q\in\mathbb{N}$ defined in Equation \eqref{omegadefinition} then:
\begin{enumerate}
\item $B_{p}\ast \omega^q(x)=B_{p+q}(x).$ \label{propconv1}
\item $\omega_h^p\ast \omega_h^q(x)=\omega^{p+q}_h(x).$ \label{propconv2}
\item $\left(B_{p} \left(\frac{\cdot}{h}-n\right)\ast \omega^q_h\right)(x)=B_{p+q}\left(\frac{x}{h}-n\right).$ \label{propconv3}
\end{enumerate}
\end{lemma}
\begin{proof}
\begin{enumerate}
\item From
$$B_{p}\ast \omega^1(x)=B_{p}\ast\chi_{[-\frac{1}{2},\frac{1}{2}]}(x)=B_{p+1}(x),$$
we get
$$B_{p}\ast \omega^q(x)=B_{p+q}(x).$$
\item From
\begin{equation}
\begin{split}
\omega_h^1\ast \omega_h^1(x)&=\frac{1}{h^2}\int_{\mathbb{R}} \chi_{[-\frac{1}{2},\frac{1}{2}]}\left(\frac{y}{h}\right)\chi_{[-\frac{1}{2},\frac{1}{2}]}\left(\frac{x-y}{h}\right)dy=\frac{1}{h}\int_{\mathbb{R}} \chi_{[-\frac{1}{2},\frac{1}{2}]}\left(z\right)\chi_{[-\frac{1}{2},\frac{1}{2}]}\left(\frac{x}{h}-z\right)dz\\
&=\frac{1}{h}\omega^2\left(\frac{x}{h}\right)=\omega^2_h(x),
\end{split}
\end{equation}
we get \ref{propconv}.\ref{propconv2}.
\item From
\begin{equation}
\begin{split}
\left(B_{p} \left(\frac{\cdot}{h}-n\right)\ast \omega^1_h\right)(x)&=\int_{\mathbb{R}} B_p\left(\frac{y}{h}-n\right)\frac{1}{h}\chi_{[-\frac{1}{2},\frac{1}{2}]}\left(\frac{x-y}{h}\right)dy=\int_{\mathbb{R}} B_p(z)\chi_{[-\frac{1}{2},\frac{1}{2}]}\left(\frac{x}{h}-n-z\right)dz\\
&=B_{p+1}\left(\frac{x}{h}-n\right),\\
\end{split}
\end{equation}
we have
$$\left(B_{p} \left(\frac{\cdot}{h}-n\right)\ast \omega^q_h\right)(x)=B_{p+q}\left(\frac{x}{h}-n\right).$$
\end{enumerate}
\end{proof}
With these results we have all the tools needed to construct the function $\Theta^q_p$ and the associated operator $Q^q_p$.
\subsubsection{Main result and examples}
We start proving a relation between the evaluation of the function $L_p$ for point-value data and $L_{p+q}$ for $q$-local average data, $f^q$.
\begin{theorem}\label{teoremaimportante}
Let $L_p$ be the operator defined in Equation \eqref{operadorL}, $h>0$, $n\in\mathbb{Z}$, $q\in \mathbb{N}$ and $f\in \Pi^p(\mathbb{R})$ then
\begin{equation}
L_{p+q}(f^q_{n,p+q})=L_{p}\left(f_{n,p}\right).
\end{equation}
\end{theorem}
\begin{proof}
If $f\in\Pi^p(\mathbb{R})$ by Proposition \ref{quasiintprop} we have
\begin{equation*}
f(x)=\sum_{n\in\mathbb{Z}}L_{p}\left(f_{n-\left\lfloor \frac{p}{2} \right\rfloor},\hdots,f_{n+\left\lfloor \frac{p}{2} \right\rfloor}\right)B_{p}\left(\frac{x}{h}-n\right),
\end{equation*}
and if we convolve $f$ with $\omega^q_h$ and use Lemma \ref{propconv} we get
\begin{equation}\label{f}
\begin{split}
f\ast \omega^q_h(x)&=\left(\left(\sum_{n\in\mathbb{Z}}L_{p}\left(f_{n-\left\lfloor \frac{p}{2} \right\rfloor},\hdots,f_{n+\left\lfloor \frac{p}{2} \right\rfloor}\right)B_{p} \left(\frac{\cdot}{h}-n\right)\right)\ast \omega^q_h\right)(x)\\
&=\sum_{n\in\mathbb{Z}}L_{p}\left(f_{n-\left\lfloor \frac{p}{2} \right\rfloor},\hdots,f_{n+\left\lfloor \frac{p}{2} \right\rfloor}\right)B_{p+q}\left(\frac{x}{h}-n\right).\\
\end{split}
\end{equation}
If $f\in\Pi^p(\mathbb{R})$ then $f\ast \omega_h^q\in \Pi^{p}(\mathbb{R})$ and by Proposition \ref{quasiintprop}:
\begin{equation}\label{fw}
f\ast \omega_h^{q}(x)=\sum_{n\in\mathbb{Z}}L_{p+q}(f^q_{n-\left\lfloor \frac{p+q}{2} \right\rfloor},\hdots,f^q_{n+\left\lfloor \frac{p+q}{2} \right\rfloor})B_{p+q}\left(\frac{x}{h}-n\right).
\end{equation}
By \eqref{f}, \eqref{fw} and from the fact that $B_{p+q}$ are linearly independent we get the result.
\end{proof}
Note that if $q=1$ then we have the relation between cell-average and point-values.
\begin{corollary}\label{corchulo}
Let $L_p$ be the operator defined in Equation \eqref{operadorL}, $h>0$, $n\in\mathbb{Z}$, $q_1,q_2\in \mathbb{N}$ and $f\in \Pi^p(\mathbb{R})$ then
\begin{equation}
L_{p+q_1}(f^{q_1}_{n,p+q_1})=L_{p+q_2}(f^{q_2}_{n,p+q_2}).
\end{equation}
\end{corollary}
\begin{proof}
Let $q_1,q_2$ be two constants, and $f\in \Pi^p(\mathbb{R})$; by Theorem \ref{teoremaimportante} we obtain:
\begin{equation}
\begin{split}
L_{p+q_1}(f^{q_1}_{n,p+q_1})&=L_{p}\left(f_{n,p}\right)=L_{p+q_2}(f^{q_2}_{n,p+q_2}).
\end{split}
\end{equation}
\end{proof}
We introduce some examples to clarify this main result.
\begin{example} We have chosen some specific values for the parameters $p$ and $q$. If $q=1$, the cell-average and point-value relation is obtained.
\begin{itemize}
\item For $p=2$, and $q=1$, we suppose that $f\in \Pi^{2}(\mathbb{R})$. We get:
\begin{equation*}
\begin{split}
L_3(\bar f_{n-1},\bar f_n,\bar f_{n+1})&=-\frac{1}{6}\bar f_{n-1}+\frac{4}{3}\bar f_n-\frac{1}{6}\bar f_{n+1}\\
&=-\frac18f_{n-1}+\frac54f_{n}-\frac18f_{n+1}\\
&=L_2(f_{n-1},f_{n},f_{n+1}).
\end{split}
\end{equation*}
\item For $p=3$ and $q=1$, we suppose that $f\in \Pi^{3}(\mathbb{R})$ and the result is:
\begin{equation*}
\begin{split}
L_4(\bar f_{n-2},\bar f_{n-1},\bar f_n,\bar f_{n+1},\bar f_{n+2})&=\frac{47}{1152}(\bar f_{n-2}+\bar f_{n+2})-\frac{107}{288}(\bar f_{n-1}+ \bar f_{n+1})+\frac{319}{192}\bar f_{n}\\
&=-\frac{1}{6} f_{n-1}+\frac{4}{3}f_n-\frac{1}{6}f_{n+1}\\
&=L_3(f_{n-1},f_{n},f_{n+1}).
\end{split}
\end{equation*}
\item For $p=3$ and $q=2$, we suppose that $f\in \Pi^{3}(\mathbb{R})$, we obtain:
\begin{equation*}
\begin{split}
L_5(  f^2_{n-2},  f^2_{n-1},  f^2_n,  f^2_{n+1},  f^2_{n+2})&=\frac{13}{240}(  f^2_{n-2}+  f^2_{n+2})-\frac{7}{15}(  f^2_{n-1}+  f^2_{n+1})+\frac{73}{40}  f^2_{n}\\
&=-\frac{1}{6} f_{n-1}+\frac{4}{3}f_n-\frac{1}{6}f_{n+1}\\
&=L_3(f_{n-1},f_{n},f_{n+1}).
\end{split}
\end{equation*}
\item For $p=5$ and $q=2$, we suppose that $f\in \Pi^{5}(\mathbb{R})$, we have:
\begin{equation*}
\begin{split}
L_7(  f^2_{n-3},  f^2_{n-2},  f^2_{n-1},  f^2_n,  f^2_{n+1},  f^2_{n+2},  f^2_{n+3})=&-\frac{311}{15120}(  f^2_{n-3}+  f^2_{n+3})+\frac{22}{105}(  f^2_{n-2}+  f^2_{n+2})\\
&-\frac{1657}{1680}(  f^2_{n-1}+   f^2_{n+1})+\frac{2452}{945}  f^2_{n}\\
=&\frac{13}{240}(f_{n-2}+f_{n+2})-\frac{7}{15}(f_{n-1}+f_{n+1})+\frac{73}{40}f_{n}.\\
\end{split}
\end{equation*}
\end{itemize}
\end{example}
\subsection{The new quasi-interpolation operators}
In view of Theorem \ref{teoremaimportante}, we define the operator $Q^q_p$, Equation \eqref{operatorbarQ}  with $\Theta^q_p=L_{p+q}$ and $\xi_q(p)=p+q$, i.e.
\begin{equation}\label{operatorbarQfinal}
Q^q_p(f)(x)=\sum_{n\in \mathbb{Z}}L_{p+q}(f^q_{n,p+q})B_p\left(\frac{x}{h}-n \right),
\end{equation}

We summarize all the results in the following corollaries.
\begin{corollary}\label{cor1}
Let $L_p$ be the operator defined in Equation \eqref{operadorL}, $h>0$, $P\in\Pi^p(\mathbb{R})$ and let $Q_p$ and $Q^q_p$ be the operators defined in Equations \eqref{operatorQ} and \eqref{operatorbarQfinal} respectively. Then:
\begin{equation}
Q^q_{p}(P)=Q_{p}(P).
\end{equation}
\end{corollary}
\begin{proof}
By Theorem \ref{teoremaimportante}, we have:
\begin{equation*}
\begin{split}
Q^q_p(P)(x)&=\sum_{n\in \mathbb{Z}}L_{p+q}(P^q_{n-\left\lfloor \frac{p+q}{2} \right\rfloor},\hdots,P^q_{n+\left\lfloor \frac{p+q}{2} \right\rfloor})B_p\left(\frac{x}{h}-n \right)=\sum_{n\in \mathbb{Z}}L_{p}(P_{n-\left\lfloor \frac{p}{2} \right\rfloor},\hdots,P_{n+\left\lfloor \frac{p}{2} \right\rfloor})B_p\left(\frac{x}{h}-n \right)=Q_p(P)(x).
\end{split}
\end{equation*}
\end{proof}
\begin{corollary}\label{quasiintprop2}
Consider a $p$ degree B-spline, $B_{p}$ with knots $S_p$, Equation \eqref{knots}, then the local operator $Q^q_{p}$ defined in Equation \eqref{operatorbarQ} with $\Theta^q_p=L_{p+q}$ and $\xi_q(p)=p+q$ on $q$-average data reproduces $\Pi^{p}(\mathbb{R})$.
\end{corollary}
\begin{proof}
Let $P\in\Pi^p(\mathbb{R})$, then by Corollary \ref{cor1}:
$$P(x)=Q_p(P)(x)=Q^q_p(P)(x).$$
\end{proof}

\begin{theorem}\label{quasiintprop3}
Let $h>0$, $N,p,q\in \mathbb{N}$ be some constants, $B_p$ be a $p$ degree B-spline with knots $S_p$, Equation \eqref{knots}, the local operator $Q^q_{p}$ defined in Equation \eqref{operatorbarQfinal} on $q$-average data and let $\Omega=[-Nh,Nh]$ be a closed interval. Considering $f\in\mathcal{C}^{p+1}(\tilde\Omega)$, with
$\tilde\Omega=[-Mh,Mh],$ being $$M=\frac{q}{2}+N+\frac{p+1}{2}+\left\lfloor \frac{p+q}{2}\right\rfloor,$$
assume we are given the $q$-average data of $f$, $f^q_n$, on $\tilde\Omega$.
Then
$$\max\{|f(x)-Q^q_{p}(f)(x)|:x\in\Omega\}=:||f-Q^q_{p}(f)||_{\infty,\Omega} \leq ((p+2)\alpha||L_{p+q}||_\infty+1)C h^{p+1},$$
where  $$C=\frac{||f^{p+1)}||_{\infty,\tilde \Omega}}{(p+1)!}\quad  \text{and} \quad \alpha=\left(\frac{q}{2}+\frac{p+1}{2}+\left\lfloor \frac{p+q}{2}\right\rfloor\right)^{p+1}.$$
\end{theorem}
\begin{proof}
Let $x_0\in \Omega$. There exists $n_0\in\mathbb{Z}$, with $|n_0|<N$ such that:
\begin{equation}\label{cota0}
(n_0-1)h\leq x_0 \leq n_0 h \to |x-n_0h|\leq h,
\end{equation}
also, note that since the support of $B_p$ is $I_p=[-\frac{p+1}{2},\frac{p+1}{2}]$, Equation \eqref{Ip}, we have:
\begin{equation}\label{interval}
B_p\left(\frac{x_0}{h} -n\right) =B_p\left(\frac{n_0h}{h} -n\right)= B_p\left(n_0 -n\right)\neq 0 \Leftrightarrow n\in\left[n_0-\frac{p+1}{2},n_0+\frac{p+1}{2}\right].
\end{equation}
From $f\in\mathcal{C}^{p+1}(\tilde \Omega)$, by Taylor's expansion there exists $\xi \in \tilde \Omega$ such that
$$R(x):=f(x)-P(x)=\frac{f^{p+1)}(\xi)}{(p+1)!}(x-n_0h)^{p+1},$$
and $P\in\Pi^p(\mathbb{R})$. As $\tilde \Omega$ is compact, there exists $C>0$ such that
$$|R(x)|\leq C |x-n_0h|^{p+1}.$$
If $j \in \left[n_0-\frac{p+1}{2}-\left\lfloor \frac{p+q}{2}\right\rfloor,n_0+\frac{p+1}{2}+\left\lfloor \frac{p+q}{2}\right\rfloor\right]$ and $ z\in [-\frac{q}{2},\frac{q}{2}]$ from
$$|z+j-n_0|^{p+1}\leq \left(\frac{q}{2}+\frac{p+1}{2}+\left\lfloor \frac{p+q}{2}\right\rfloor\right)^{p+1}=:\alpha,$$
we have that,
\begin{equation}\label{cota1}
|R(hz+hj)|\leq C|hz+hj-n_0h|^{p+1} =Ch^{p+1}|z+j-n_0|^{p+1}\leq \alpha C h^{p+1},
\end{equation}
and by Equation \eqref{defdiscretizacion} and \eqref{cota1}, we get
\begin{equation}\label{cota2}
\begin{split}
f_j^q-P_j^q&=\left\langle f, \frac{1}{h}\omega^q\left(\frac{\cdot}{h}-j \right)\right\rangle-\left\langle P, \frac{1}{h}\omega^q\left(\frac{\cdot}{h}-j \right)\right\rangle=\left\langle R, \frac{1}{h}\omega^q\left(\frac{\cdot}{h}-j \right)\right\rangle=R_j^q\\
&=\frac{1}{h}\int_{\mathbb{R}} R(x)\omega^q\left(\frac{x}{h}-j\right)dx =\int_{\mathbb{R}} R(hz+hj)\omega^q\left(z\right)dz\\
&=\int_{\mathbb{R}} R(hz+hj)B_{q-1}\left(z\right)dz=\int_{-\frac q2}^{\frac q2} R(hz+hj)B_{q-1}\left(z\right)dz\\
|R_j^q|&\leq \alpha C h^{p+1}\int_{-\frac q2}^{\frac q2} B_{q-1}\left(z\right)dz = \alpha C h^{p+1}.
\end{split}
\end{equation}
If $n\in [n_0-\frac{p+1}{2},n_0+\frac{p+1}{2}]$, by Def. \ref{def:normaL} and \eqref{cota2}:
\begin{equation}\label{cota3}
|L_{p+q}(f_{n,p+q})-L_{p+q}(P_{n,p+q})|=|L_{p+q}(R_{n,p+q})|\leq\sum_{j=-\left\lfloor \frac{p+q}{2}\right\rfloor}^{\left\lfloor \frac{p+q}{2}\right\rfloor} |c_{p+q,j}||R^q_{n+j}|\leq \alpha C||L_{p+q}||_\infty h^{p+1}.
\end{equation}
By \eqref{interval} and \eqref{cota3}
\begin{equation}\label{cota4}
\begin{split}
|Q^q_p(f)(x_0)-Q^q_{p}(P)(x_0)| = |Q^q_p(R)(x_0)| &=\left|\sum_{n\in\mathbb{Z}}L_{p+q}(R^q_{n,p+q})B_{p}\left(\frac{x_0}{h}-n\right) \right|\\
&\leq \sum_{n=\left\lceil n_0 -\frac{p+1}{2}\right\rceil}^{\left\lfloor n_0+\frac{p+1}{2}\right\rfloor}|L_{p+q}(R^q_{n,p+q})|\left|B_{p}\left(\frac{x_0}{h}-n\right) \right|\\
&\leq (p+2)\alpha C ||L_{p+q}||_\infty h^{p+1}.
\end{split}
\end{equation}
By Proposition \ref{quasiintprop} we have that $Q_p(P)=P$, by Corollary \ref{quasiintprop2}
$Q^q_p(P)=Q_p(P)$ and using \eqref{cota4} we obtain:
\begin{equation}
\begin{split}
|Q^q_{p}(f)(x_0)-f(x_0)|&=|Q^q_{p}(f)(x_0)-Q_p(P)(x_0)+P(x_0)-f(x_0)|\\
&=|Q^q_p(f)(x_0)-Q^q_{p}(P)(x_0)+P(x_0)-f(x_0)|\\
&\leq|Q^q_p(R)(x_0)|+|R(x_0)|\\
&\leq ((p+2)\alpha||L_{p+q}||_\infty+1)C h^{p+1}.
\end{split}
\end{equation}
\end{proof}

With this result, we have a way to obtain an approximation of the function $f$ with an approximation order $O(h^{p+1})$ using the $q$-average data, for example, cell-average. In practice, we will have the data from $q$-average discretisations in an interval, and we need to know the $q$-average in the extended interval $\tilde \Omega$. As a matter of fact, it is enough to know these values within $O(h^{p+1})$ accuracy. Since $f\in \mathcal{C}^{p+1}(\Omega)$, we suggest extending the given values  using a simple polynomial extrapolation. We rely here on an assumption that the function $f$ can be extended smoothly. This assumption is valid by using Whitney's extension theorem \cite{whitney}.

\subsection{Multivariate quasi-interpolators from $q$-average data}
In this section, we suppose a real function $f:\mathbb{R}^k\to \mathbb{R}$,  $\mathbf{p}=(p_1,\hdots,p_k)$ parameters defining the polynomials in Equation \eqref{polinomiosvarias},
$\Pi^{\mathbf{p}}(\mathbb{R}),$ we let $h>0$, $\mathbf{n}=(n_1,\hdots,n_k)$  and consider the values
$$\{f_{\mathbf{n}}=f({n_{1} h,\hdots,n_{k}h})\}.$$
We define
$$f_{\mathbf{n},\mathbf{p}}=\left\{f_{(n_{1}+j_1,\hdots,n_{k}+j_k)}:0\leq |j_l|\leq \left\lfloor \frac{p_l}{2}\right\rfloor, l=1,\hdots,k\right\}.$$
Note that if $k=1$, $f_{n,p}=(f_{n-\left\lfloor \frac{p}{2} \right\rfloor},\hdots,f_{n+\left\lfloor \frac{p}{2} \right\rfloor})$.
We consider the following operator which is the tensor product of $L_p$, Equation \eqref{operadorL},
\begin{equation}\label{operadorLmulti}
L_{\mathbf{p}}(f_{\mathbf{n},\mathbf{p}})=\sum_{j_1=-\left\lfloor \frac{p_1}{2} \right\rfloor}^{\left\lfloor \frac{p_1}{2} \right\rfloor}\hdots\sum_{j_k=-\left\lfloor \frac{p_k}{2} \right\rfloor}^{\left\lfloor \frac{p_k}{2} \right\rfloor} c_{p_1,j_1}\hdots c_{p_k,j_k} f_{n_1+j_1,\hdots,n_k+j_k},
\end{equation}
where $c_{p,j}$ are the values defined in Equation \eqref{coeficientesc}. Finally, if $\mathbf{x}=(x_1,\hdots,x_k)$ and $\mathbf{q}=(q,\hdots,q)$ then we define the tensor products:
$$B_{\mathbf{p}}\left(\frac{\mathbf{x}}{h}-\mathbf{n}\right)=\prod_{l=1}^kB_{p_{l}}\left(\frac{x_l}{h}-n_l\right), \quad \omega^{\mathbf{q}}\left(\mathbf{x}\right)=\prod_{l=1}^k\omega^q(x_l).$$
With these elements it is well-known that the operator:
\begin{equation}\label{operatorQmulti}
Q_{\mathbf{p}}(f)(\mathbf{x})=\sum_{\mathbf{n}\in\mathbb{Z}^k} L_{\mathbf{p}}(f_{\mathbf{n},\mathbf{p}})B_{\mathbf{p}}\left(\frac{\mathbf{x}}{h}-\mathbf{n}\right),
\end{equation}
reproduces tensor product polynomials of degrees up to $\mathbf{p}$, \cite{speleers}. Thus, if we denote as
$$f^q_{\mathbf{n}}=f\ast\omega_h^{\mathbf{q}}(\mathbf{n}h),$$
with
$$\omega_h^{\mathbf{q}}=\frac{1}{h^k}\omega^q\left(\frac{\mathbf{x}}{h}\right),$$
we can present the following corollary.
\begin{corollary}\label{cormulti}
We consider $\mathbf{p} \in \mathbb{N}^k$, $q\in \mathbb{N}$, $n\in\mathbb{Z}$, the operator $L_{\mathbf{p}}$ defined in Equation \eqref{operadorLmulti} and $f\in \Pi^{\mathbf{p}}(\mathbb{R})$, then:
\begin{equation}
L_{\mathbf{p}+\mathbf{q}}(f^q_{\mathbf{n},\mathbf{p}+\mathbf{q}})=L_{\mathbf{p}}(f^q_{\mathbf{n},\mathbf{p}}),
\end{equation}
being $\mathbf{q}=(q,\hdots,q)$.
\end{corollary}
\begin{proof}
It follows directly from by Theorem \ref{teoremaimportante}.
\end{proof}

If we define
\begin{equation}\label{Qqmulti}
Q^q_{\mathbf{p}}(f)(\mathbf{x})=\sum_{\mathbf{n}\in\mathbb{Z}^k} L_{\mathbf{p}+\mathbf{q}}(f^q_{\mathbf{n},\mathbf{p}+\mathbf{q}})B_{\mathbf{p}}\left(\frac{\mathbf{x}}{h}-\mathbf{n}\right),
\end{equation}
a direct consequence of Corollary \ref{cormulti} is the following corollary.
\begin{corollary}
Let $Q_{\mathbf{p}}$ and $Q^q_{\mathbf{p}}$ be  the operators  defined in Equation \eqref{operatorQmulti} and \eqref{Qqmulti} with $\mathbf{p} \in \mathbb{N}^k$, $q\in \mathbb{N}$ and $f\in \Pi^{\mathbf{p}}(\mathbb{R})$, then:
\begin{equation}
Q_{\mathbf{p}}(f)=Q^q_{\mathbf{p}}(f),
\end{equation}
where $\mathbf{q}=(q,\hdots,q)$.
\end{corollary}
\begin{example}
We consider a simple example when $\mathbf{p}=(2,2)$ and $q=1$, ($f^1=\bar{f}$) then
\begin{equation}
\begin{split}
L_{(3,3)}(\bar{f}_{(n_1,n_2),(3,3)})=&\frac{16}{9}\bar f_{(n_1,n_2)} -\frac29(\bar f_{(n_1-1,n_2)}+\bar f_{(n_1+1,n_2)}+\bar f_{(n_1,n_2-1)}+\bar f_{(n_1,n_2+1)})\\
&+\frac{1}{36}(\bar f_{(n_1-1,n_2-1)}+\bar f_{(n_1+1,n_2+1)}+\bar f_{(n_1+1,n_2-1)}+\bar f_{(n_1-1,n_2+1)}) \\
=&\frac{25}{16}f_{(n_1,n_2)} -\frac{5}{32}(f_{(n_1-1,n_2)}+f_{(n_1+1,n_2)}+f_{(n_1,n_2-1)}+f_{(n_1,n_2+1)})\\
&+\frac{1}{64}(f_{(n_1-1,n_2-1)}+f_{(n_1+1,n_2+1)}+f_{(n_1+1,n_2-1)}+f_{(n_1-1,n_2+1)}) \\
=&L_{(2,2)}({f}_{(n_1,n_2),(2,2)}).
\end{split}
\end{equation}
\end{example}
Finally, we summarise this result in the following corollary.
\begin{corollary}
Consider a $\mathbf{p}=(p_1,\hdots,p_k)$ degree multi B-spline, $B_{\mathbf{p}}$, the tensor product of the $B_{p_l}$, $1\leq l\leq k$, with knots $S_{p_l}$ Equation \eqref{knots}, then the local operator $Q^{q}_{\mathbf{p}}$ defined in Equation \eqref{Qqmulti}  reproduces $\Pi^{\mathbf{p}}(\mathbb{R})$ and achieves $O(h^{\min_{1\leq l \leq k}\{p_l\}+1})$ approximation order.
\end{corollary}

\section{Conclusions}

In this work, two main results are presented: We give an explicit method to obtain the values of a function from cell-average data without obtaining the coefficients of the polynomials, we prove the error bound and extend the result to several dimensions. In the second part of the paper, we introduce B-splines theory and a generalisation of the cell-average discretisation using $q$-average based on the repeat convolution of the characteristic function in the interval $[-\frac12,\frac12]$. We prove a relation between all the $q$-discretisations, offering a simple algorithm to switch between them. As a particular case, we give a formula to obtain an approximation of a function from cell-average data. In all these cases, proofs for the accuracy order are given and some examples are performed. Finally, we extend these results to any dimension using a tensor product.

\bibliographystyle{these}

\begin{thebibliography}{100}

\bibitem{IMAJNA} {\sc S. Amat, D. Levin, J. Ruiz,} (2021). {\em A two-stage approximation strategy for piecewise smooth functions in 2-D and 3-D}. IMA J. Numer. Anal.


\bibitem{ACDD} {\sc F. Ar\`andiga, A. Cohen, R. Donat, N. Dyn,} (2005).{\em Interpolation and approximation of piecewise smooth functions.} SIAM J. Numer. Anal. 43  no. 1, 41-57.

\bibitem{AD} {\sc  F. Aràndiga, R. Donat,} (2000). {\em Nonlinear multiscale decompositions: The approach of A. Harten.} {\em Numerical Algorithms 23}, 175-216.


\bibitem{butzeretal}
{\sc P.L. Butzer, K. Schmidt, E. L. Stark and L. Vogt,} (1989): ``Central factorial numbers; their main properties and
some applications'', {\em Numer. Funct. Anal. Optim. 10}, 419–488.


\bibitem{donohoetal}
{\sc D. Donoho, N.Dyn, D. Levin, T. P. Yu,} (2000): ``Smooth Multiwavelet Duals of Alpert Bases by Moment-Interpolating Refinemen'', {\em Appl.  Comput. Harmon. Anal. 9(2)} 166--203.

\bibitem{harten}
{\sc A. Harten,} (1996): ``Multiresolution representation of data: General framework'', {\em SIAM J. Numer.
Anal., 33} 1205--1256.

\bibitem{lycheetal}
{\sc T. Lyche, C. Manni and H. Speleers}, (2017):   \href{https://www.mat.uniroma2.it/~speleers/cime2017/material/notes_lyche.pdf}{``B-Splines and Spline Approximation''}.



\bibitem{speleers}
{\sc H. Speleers,} (2017):  ``Hierarchical spline spaces: quasi-interpolants and local approximation estimates'', {\em Adv. Comput. Math.,  43}, 235-255.

\bibitem{SHU_cell}{\sc C-W. Shu,} (1998). {\em Essentially non-oscillatory and weighted essentially non-oscillatory schemes for hyperbolic conservation laws.} In: Quarteroni A. (eds) {\em Advanced Numerical Approximation of Nonlinear Hyperbolic Equations. Lecture Notes in Mathematics, vol 1697}. Springer, Berlin, Heidelberg.

\bibitem{whitney}
{\sc H. Whitney} (1934): ``Analytic extensions of differentiable functions defined in closed sets'', {\em Transactions of the American Mathematical Society, American Mathematical Society, 36 (1)}.


\end{thebibliography}

\appendix

\section{From cell-average to point values: 2D case}\label{seccion2d}

In this appendix, we suppose that $m^{x},m^{y},p\in\mathbb{N}$ with $2m^\mu+2\leq p$, $\mu=x,y$; $x_1,x_2,y_1,y_2,h_x,h_y\in\mathbb{R}$ with $x_1<x_2$, $y_1<y_2$, $h_x,h_y>0$ and $f\in \mathcal{C}^p([x_1,x_2]\times[y_1,y_2])$ with $[a-m^xh_x,a+m^xh_x]\times [b-m^yh_y,b+m^yh_y] \subset [x_1,x_2]\times[y_1,y_2]$.
Also, we consider $\tilde\Omega=[a-\frac {h_x}{2},a+\frac{h_x}{2}]\times [b-\frac {h_y}{2},b+\frac{h_y}{2}]$ and
$$\bar f(a,b)= \frac{1}{h_xh_y}\int_{\tilde\Omega}f(x,y)d(x,y).$$
Since $f\in \mathcal{C}^p(\tilde\Omega)$ then:
$$\int_{\tilde \Omega}f(x,y)d(x,y)=\int_{a-\frac{h_x}{2}}^{a+\frac{h_x}{2}}\int_{b-\frac{h_y}{2}}^{b+\frac{h_y}{2}}f(x,y)dydx=\int_{b-\frac{h_y}{2}}^{b+\frac{h_y}{2}}\int_{a-\frac{h_x}{2}}^{a+\frac{h_x}{2}}f(x,y)dxdy.$$
We define the function:
$$F(x)=\frac{1}{h_y}\int_{b-\frac{h_y}{2}}^{b+\frac{h_y}{2}} f(x,y)dy.$$
By Leibnitz rule, $F\in \mathcal{C}^p([x_1,x_2])$, then following the analysis in Section \ref{generalcentro} we get:
\begin{equation}\label{eq71dgeneral}
\frac{1}{h_x}\int_{a-\frac{h_x}{2}}^{a+\frac{h_x}{2}} F(x)dx+\sum_{i=1}^{m^x}a_{i}\Delta^{2i}\bar{F}(a)= F(a)+O(h_x^{2m^x+1}).
\end{equation}
We denote
$$F(a)=\frac{1}{h_y}\int_{b-\frac{h_y}{2}}^{b+\frac{h_y}{2}} f(a,y)dy=\frac{1}{h_y}\int_{b-\frac{h_y}{2}}^{b+\frac{h_y}{2}} f_{a}(y)dy=\bar{f}_{a}(b),$$
noting that $f_{a} \in \mathcal{C}^p([y_1,y_2])$. Using the same formula for $\bar{f}_{a}(b)$, we get:
\begin{equation}\label{eq72dgeneral}
\bar{f}_{a}(b)+\sum_{j=1}^{m^y}a_{j}\Delta^{2j}\bar{f}_{a}(b)=f_{a}(b)+O(h_y^{2m^y+1})=f(a,b)+O(h_y^{2m^y+1}).
\end{equation}
Therefore, by Equations \eqref{eq71dgeneral} and \eqref{eq72dgeneral} we have that:
\begin{equation}\label{ecuacion1general}
\begin{split}
\bar f(a,b) +\sum_{i=1}^{m^x}a_{i}\Delta^{2i}\bar{F}(a)+ \sum_{j=1}^{m^y}a_{j}\Delta^{2j}\bar{f}_{a}(b)&=F(a)+\sum_{j=1}^{m^y}a_{j}\Delta^{2j}\bar{f}_{a}(b)+O(h_x^{2m^x+1})\\
&= f(a,b)+O(h_y^{2m^y+1})+O(h_x^{2m^x+1}).
\end{split}
\end{equation}
For $i\in\mathbb{N}$ we denote:
\begin{equation}\label{ecuacion2general}
\Delta^{2i}\bar{F}(a)=\Delta^{2i,x}\bar{f}(a,b).
\end{equation}
Subsequently, we know that for all $y\in[y_1,y_2]$ we have by Section \ref{sec1} that there exists $\xi \in[x_1,x_2]$ such that:
\begin{equation*}
f_a(y)=f(a,y)=\frac{1}{h_x}\int_{a-\frac{h_x}{2}}^{a+\frac{h_x}{2}}f(x,y)dx+\sum_{i=1}^{m^x}a_i \Delta^{2i}\bar{f}^y(a)+K\frac{\partial^{2m^x+2}f}{\partial x^{2m^x+2}}(\xi,y)h_x^{2m^x+2},
\end{equation*}
where $K$ is a constant independent on $y$; $f^y(x):=f(x,y)$ and
$$\bar{f}^y(a)=\frac{1}{h_x}\int_{a-\frac{h_x}{2}}^{a+\frac{h_x}{2}}f(x,y)dx,$$
Note that  $\frac{\partial^{2m^x+2}f}{\partial x^{2m^x+2}}(\xi,y)$ is continuous in $[y_1,y_2]$ since $f\in \mathcal{C}^p([x_1,x_2]\times[y_1,y_2])$ (so $\frac{\partial^{2m^x+2}f}{\partial x^{2m^x+2}}(\xi,y)\in\mathcal{L}^1([y_1,y_2])$).
Thus,
\begin{equation}
\begin{split}
\bar{f}_{a}(b)&=\frac{1}{h_y}\int_{b-\frac{h_y}{2}}^{b+\frac{h_y}{2}} f_{a}(y)dy \\
&= \frac{1}{h_y}\int_{b-\frac{h_y}{2}}^{b+\frac{h_y}{2}}
\left(\frac{1}{h_x}\int_{a-\frac{h_x}{2}}^{a+\frac{h_x}{2}}f(x,y)dx+\sum_{i=1}^{m^x}a_i \Delta^{2i}\bar{f}^y(a)+K\frac{\partial^{2m^x+2}f}{\partial x^{2m^x+2}}(\xi,y)h_x^{2m^x+2}\right)dy\\
&= \frac{1}{h_y}\int_{b-\frac{h_y}{2}}^{b+\frac{h_y}{2}}
\left(\frac{1}{h_x}\int_{a-\frac{h_x}{2}}^{a+\frac{h_x}{2}}f(x,y)dx+\sum_{i=1}^{m^x}a_i \Delta^{2i}\bar{f}^y(a)\right)dy+O(h_x^{2m^x+2})\\
&= \bar{f}(a,b)+\sum_{i=1}^{m^x}a_i \Delta^{2i,x}\bar{f}(a,b)+O(h_x^{2m^x+2}).
\end{split}
\end{equation}
Now we can write that,
\begin{equation}\label{ecuacion3general}
\Delta^{2j}\bar{f}_{a}(b)=\Delta^{2j,y}\bar{f}(a,b)+\sum_{i=1}^{m^x}a_i \Delta^{2j,y}\Delta^{2i,x}\bar{f}(a,b)+O(h_x^{2m^x+1}).
\end{equation}
Collecting Equations \eqref{ecuacion1general}, \eqref{ecuacion2general} and \eqref{ecuacion3general}, we get
\begin{equation}\label{ecuacion4general}
\begin{split}
f(a,b)=&\bar f(a,b) +\sum_{i=1}^{m^x}a_{i}\Delta^{2i}\bar{F}(a)+ \sum_{j=1}^{m^y}a_{j}\Delta^{2j}\bar{f}_{a}(b)+O(h_y^{2m^y+1})+O(h_x^{2m^x+1})\\
=&\bar{f}(a,b)+\sum_{i=1}^{m^x}a_{i}\Delta^{2i,x}\bar{f}(a,b)+\sum_{j=1}^{m^y}a_{j}\Delta^{j,y}\bar{f}(a,b)\\
&+\sum_{j=1}^{m^y}\sum_{i=1}^{m^x}a_{j}a_{i}\Delta^{2j,y}\Delta^{i,x}\bar{f}(a,b)+O(h_y^{2m^y+1})+O(h_x^{2m^x+1})\\
=&\sum_{j=0}^{m^y}\sum_{i=0}^{m^x}a_{j}a_{i}\Delta^{j,y}\Delta^{2i,x}\bar{f}(a,b)+O(h_y^{2m^y+1})+O(h_x^{2m^x+1}).
\end{split}
\end{equation}
with $a_0=1$.
We perform the following example in order to clarify the ideas proposed in this appendix.
\begin{example}
For example, if $m^x=m^y=1$ and $h=h_x=h_y$, we know by Section \ref{sec1} that $a_1=-1/24$, then:
\begin{equation}\label{ecuacion5}
\begin{split}
f(a,b)+O(h^{4})=&\bar{f}(a,b)-\frac{1}{24}(\Delta^{2,x}\bar{f}(a,b)+\Delta^{2,y}\bar{f}(a,b))+\frac{1}{576}\Delta^{2,y}\Delta^{2,x}\bar{f}(a,b)\\
=&\bar{f}(a,b)-\frac{1}{24}(\Delta^{2,x}\bar{f}(a,b)+\Delta^{2,y}\bar{f}(a,b))+\\
&+\frac{1}{576}(\bar{f}(a-h,b-h)-2\bar{f}(a-h,b)+\bar{f}(a-h,b+h))+\\
&+\frac{1}{576}(-2\bar{f}(a,b-h)+4\bar{f}(a,b)-2\bar{f}(a,b+h))+\\
&+\frac{1}{576}(\bar{f}(a+h,b-h)-2\bar{f}(a+h,b)+\bar{f}(a+h,b+h)).\\
\end{split}
\end{equation}
\end{example}

Again, the next result gives an exact formula for polynomials.

\begin{corollary}
Let $m^x,m^y\in\mathbb{N}$ and $P\in\Pi^{(2m^x+1,2m^y+1)}_2(\mathbb{R})$, then
$$P(a,b)=\sum_{j=0}^{m^y}\sum_{i=0}^{m^x}a_{j}a_{i}\Delta^{2j,y}\Delta^{2i,x}\bar{P}(a,b),\,\,\, \forall  \,(a,b)\in\mathbb{R}^2,$$
being $a_0=1$ and $\mathbf{a}_{m}=(a_i)_{i=1}^{m}$ the solution of the system \eqref{sys}.
\end{corollary}
\begin{proof}
In order to prove the corollary we introduce two considerations:
\begin{itemize}
\item As $P(x,y) \in \Pi^{2m^x+1,2m^y+1}_2(\mathbb{R})$, then the function
$$F(x)=\frac{1}{h_x}\int_{b-\frac{h_y}{2}}^{b+\frac{{h_y}}{2}} P(x,y)dy,$$
is a polynomial of degree $2m^x+1$ on the variable $x$, thus
\begin{equation}\label{paracor21}
\frac{\partial^{2m^x+2} F}{\partial x^{2m^x+2}}(x)=0,\,\, \forall x\,\in \mathbb{R}.
\end{equation}
\item  For all $\mu \in \mathbb{R}$, the function defined as $P_{\mu}(y)=P(\mu,y)$ is a polynomial of degree $2m^y+1$ on the variable $y$, thus
    \begin{equation}\label{paracor22}
\frac{\partial^{2m^y+2}P_{\mu}}{\partial y^{2m^y+2}}(y)=0, \,\, \forall y \,\in \mathbb{R}.
\end{equation}
\end{itemize}
With these considerations, the proof is similar to the construction of the formula \eqref{ecuacion4general} since:

By Section \ref{sec1} and Equation \eqref{paracor21} we have that the formula for the 1d case is exact, i.e.
\begin{equation}\label{eqcor1}
\frac{1}{h_x}\int_{a-\frac{h_x}{2}}^{a+\frac{h_x}{2}} F(x)dx+\sum_{i=1}^{m^x}a_{i}\Delta^{2i}\bar{F}(a)= F(a),
\end{equation}
being $\mathbf{a}_{m^x}=(a_i)_{i=1}^{m^x}$ the solution of the system in \eqref{sys}.
Again, using the same formula for $\bar{p}_a(b)$ and Equation \eqref{paracor22}:
\begin{equation}\label{eqcor2}
\bar{P}_a(b)+\sum_{j=1}^{m^y}a_{j}\Delta^{2j}\bar{P}_a(b)=P_a(b)=P(a,b).
\end{equation}
Thus, by Equations \eqref{eqcor1} and \eqref{eqcor2} we have that:
\begin{equation}\label{eqcor3}
\begin{split}
\bar P(a,b) +\sum_{i=1}^{m^x}a_{i}\Delta^{2i}\bar{F}(a)+ \sum_{j=1}^{m^y}a_{j}\Delta^{2j}\bar{P}_a(b)&=F(a)+\sum_{j=1}^{m^y}a_{j}\Delta^{2j}\bar{P}_a(b)= P(a,b).
\end{split}
\end{equation}
By $P(x,y)\in \Pi^{2m^x+1,2m^y+1}_2(\mathbb{R})$, we have for each $y\in \mathbb{R}$
\begin{equation*}
P_a(y)=P(a,y)=\frac{1}{h}\int_{a-\frac{h_x}{2}}^{a+\frac{h_x}{2}}P(x,y)dx+\sum_{i=1}^{m}a_i \Delta^{2i}\bar{P}^y(a).
\end{equation*}
Thus,
\begin{equation}
\begin{split}
\bar{P}_a(b)&=\frac{1}{h_y}\int_{b-\frac{h_y}{2}}^{b+\frac{h_y}{2}} P_a(y)dy = \frac{1}{h_y}\int_{b-\frac{h_y}{2}}^{b+\frac{h_y}{2}}
\left(\frac{1}{h_x}\int_{a-\frac{h_x}{2}}^{a+\frac{h_x}{2}}P(x,y)dx+\sum_{i=1}^{m^x}a_i \Delta^{2i}\bar{P}^y(a)\right)dy\\
&= \bar{P}(a,b)+\sum_{i=1}^{m^x}a_i \Delta^{2i,x}\bar{P}(a,b).
\end{split}
\end{equation}
Then
\begin{equation}\label{eqcor3}
\Delta^{2j}\bar{P}_a(b)=\Delta^{2j,y}\bar{P}(a,b)+\sum_{i=1}^{m^x}a_i \Delta^{2j,y}\Delta^{2i,x}\bar{P}(a,b).
\end{equation}
Finally, by Equations \eqref{eqcor1}, \eqref{eqcor2} and \eqref{eqcor3}, we obtain:
\begin{equation}\label{eqcor4}
\begin{split}
P(a,b)&=\bar P(a,b) +\sum_{i=1}^{m^x}a_{i}\Delta^{2i}\bar{F}(a)+ \sum_{j=1}^{m^y}a_{j}\Delta^{2j}\bar{P}_a(b)\\
&=\bar{P}(a,b)+\sum_{i=1}^{m^x}a_{i}\Delta^{2i,x}\bar{P}(a,b)+\sum_{j=1}^{m^y}a_{j}\Delta^{2j,y}\bar{P}(a,b)+\sum_{j=1}^{m^y}\sum_{i=1}^{m^x}a_{j}a_{i}\Delta^{2j,y}\Delta^{2i,x}\bar{P}(a,b)\\
&=\sum_{j=0}^{m^y}\sum_{i=0}^{m^x}a_{j}a_{i}\Delta^{2j,y}\Delta^{2i,x}\bar{P}(a,b),
\end{split}
\end{equation}
with $a_0=1$.
\end{proof}
Finally, we show a new example with different values $m^x$ and $m^y$.
\begin{example}
For example, if $m^x=2$, $m^y=1$ and $h=h_x=h_y$, we know by Section \ref{sec1} that $a_1=-1/24,\, a_2=3/640$, then:
\begin{equation}\label{ecuacion6}
\begin{split}
f(a,b)+O(\tilde{h}^{4})+O(h^6)=&\bar{f}(a,b)-\frac{1}{24}(\Delta^{2,x}\bar{f}(a,b)+\Delta^{2,y}\bar{f}(a,b))+\frac{3}{640}\Delta^{4,x}\bar{f}(a,b)\\
&+\frac{1}{576}\Delta^{2,y}\Delta^{2,x}\bar{f}(a,b)-\frac{1}{5120}\Delta^{2,y}\Delta^{4,x}\bar{f}(a,b).\\
\end{split}
\end{equation}
\end{example}

\end{document}